\documentclass[12pt]{article}
\usepackage{amssymb}
\usepackage{amsmath}
\usepackage{amsthm}
\usepackage{yhmath}
\usepackage{mathdots}
\usepackage{MnSymbol}
\usepackage{color}
\usepackage[left=2.5cm,top=2cm,right=2cm]{geometry}
\usepackage{enumitem}
\usepackage{nccmath}
\usepackage{blkarray}
\usepackage{multirow}
\usepackage{float}
\usepackage{booktabs}
\usepackage[square,numbers]{natbib}
\usepackage{graphicx}
\graphicspath{ {/Users/neha/Desktop/NLA P2/c2_CMA} }
\usepackage{subfigure}

\usepackage{natbib}
\usepackage[english]{babel}
\usepackage{algorithm}
\usepackage{algpseudocode}

\newtheorem{theorem}{\bf Theorem}[section]

\newtheorem{corollary}[theorem]{\bf Corollary}
\newtheorem{exam}[theorem]{\bf Example}
\newtheorem{remark}[theorem]{\bf Remark}
\newtheorem{lemma}[theorem]{\bf Lemma}

\SetLabelAlign{CenterWithParen}{\hfil(\makebox[1.0em]{#1})\hfil}
\def \R{{\mathbb R}}
\def \C{{\mathbb C}}
\def \rank{\mathrm{rank}}

\def \QR{\mathbb{Q}_\mathbb{R}}
\def \i{\textit{\textbf{i}}}
\def \j{\textit{\textbf{j}}}
\def \k{\textit{\textbf{k}}}
\newcommand\norm[1]{\left\lVert#1\right\rVert}
\newcommand{\Rn}[1]{%
	\textup{\lowercase\expandafter{\romannumeral#1}}%
}

\def \diag{\mathrm{diag}}

\newcommand{\beano}{\begin{eqnarray*}}	\newcommand{\eeano}{\end{eqnarray*}}

\newcommand{\mLabel}[1]{\mbox{$\scriptstyle{#1}$}}

\renewcommand{\thefootnote}{\fnsymbol{footnote}}

\date{\today}
\title{Algebraic technique for mixed least squares and total least squares problem in the reduced biquaternion algebra
	
	\footnotemark[2]}
\author{Sk. Safique Ahmad\footnotemark[1] \and Neha Bhadala \footnotemark[3]}

\begin{document}
	\maketitle
	\begin{abstract}
		This paper presents the reduced biquaternion mixed least squares and total least squares (RBMTLS) method for solving an overdetermined  system $AX \approx B$ in the reduced biquaternion algebra. The RBMTLS method is suitable when matrix $B$ and a few columns of matrix $A$ contain errors. By examining real representations of reduced biquaternion matrices, we investigate the conditions for the existence and uniqueness of the real RBMTLS solution and derive an explicit expression for the real RBMTLS solution. The proposed technique covers two special cases: the reduced biquaternion total least squares (RBTLS) method and the reduced biquaternion least squares (RBLS) method. Furthermore, the developed method is also used to find the best approximate solution to $AX \approx B$ over a complex field. Lastly, a numerical example is presented to support our findings.

	\end{abstract}
	
	
	
	\renewcommand{\thefootnote}{\fnsymbol{footnote}}
	
	\footnotetext[1]{
		Department of Mathematics, Indian Institute of Technology Indore, Simrol, Indore-452020, Madhya Pradesh, India. \texttt{email:safique@iiti.ac.in, safique@gmail.com}.}
	
	\footnotetext[3]{Research Scholar, Department of Mathematics, IIT Indore, Research work funded by PMRF (Prime Minister's Research Fellowship). \texttt{email:phd1901141004@iiti.ac.in, bhadalaneha@gmail.com}}

	\section{Introduction}\label{sec1}
	The concept of quaternions was introduced by Hamilton. In recent years, quaternions have been increasingly used in quantum mechanics, image processing, and signal processing \cite{adler1995quaternionic, ell2014quaternion}. The formulation of a solution procedure for these application problems often entails finding the best approximate solution to an inconsistent linear system. Various solution techniques have been used in the literature to find the best approximate solutions for quaternion matrix equations (for example, \cite{ jiang2007algebraic, zhang2023real, MR3537293, MR3406901}). Quaternions are non-commutative in nature, which limits their applicability in various applications.\\
		Segre introduced reduced biquaternions, which are commutative in nature and offer great applicability in various applications. See \cite{el2022linear, gai2023theory, pei2004commutative, pei2008eigenvalues} and references therein. As a result of commutative multiplication, several operations are simplified, and computation efficiency is enhanced. In literature, there are various papers where the advantage of reduced biquaternions over quaternions has been proved. For instance, Pei et al. \cite{pei2004commutative, pei2008eigenvalues} and Gai \cite{gai2023theory} demonstrated that reduced biquaternions outperform conventional quaternions in applications related to image and digital signal processing. The broad applicability and advantages of reduced biquaternions over quaternions have prompted a substantial number of research papers exploring theoretical and numerical aspects of reduced biquaternion matrix equations \cite{ding2021special, yuan2020hermitian,  zhang2022two, zhang2024singular}. In this paper, we learn how to determine the best approximate solutions to an overdetermined linear system 
		\begin{equation}\label{eq1.1}
			AX \approx B,
		\end{equation} 
		where $A= A_0+ A_1\i+ A_2\j+ A_3\k \in  \QR^{m \times n}$ $(m>n)$, $B= B_0+ B_1\i+ B_2\j+ B_3\k \in \QR^{m \times d}$, and $m \geq n+d$ that arises in commutative quaternionic theory. Our primary focus is on the inconsistent problem, i.e., $\mathcal{R}(B) \not \subset \mathcal{R}(A)$. The paper delves into different methods employed to tackle the linear approximation problem \eqref{eq1.1}. Among these methods, the least squares (LS) approach is a commonly used technique for finding the best approximate solution.\\
		 The multidimensional reduced biquaternion least squares (RBLS) problem can be formulated as:
	\begin{equation}\label{eq1.2}
		\min_{ X,\widehat{G}}\norm{\widehat{G}}_F \; \; \; \textrm{subject to} \; \; \; AX=B+\widehat{G}. 
	\end{equation}
	Once a minimizing $\widehat{G}$ is found, then any $X$ which solves the corrected system in \eqref{eq1.2} is called the RBLS solution.\\
	There is an underlying assumption in the RBLS method that all errors are contained in matrix $B$, and matrix $A$ is error-free. Depending on the nature of the applications, however, matrix $A$ may also be contaminated by noise. The RBLS solution does not take into account this requirement, resulting in poor results. To mitigate the effect of errors in both $A$ and $B$ total least squares (TLS) solution technique was introduced in the literature.\\
	 The multidimensional reduced biquaternion total least squares (RBTLS) problem can be formulated as:
	\begin{equation}\label{eq1.3}
		\min_{X,\widehat{E}, \widehat{G}}\norm{[\widehat{E}, \widehat{G}]}_F \; \; \; \textrm{subject to} \; \; \; (A+\widehat{E})X=B+\widehat{G}.
	\end{equation} 
	Once a minimizing $[\widehat{E}, \widehat{G}]$ is found, then any $X$ which solves the corrected system in \eqref{eq1.3} is called the RBTLS solution.\\
	 The total least squares method is widely used in system theory, signal processing, and computer algebra. However, in certain application problems, errors may be confined to the observation matrix $B$ and only a few columns of the data matrix $A$, while other columns of $A$ are known precisely. In such problems, perturbing the exactly known columns of $A$ as per RBTLS is detrimental to the accuracy of the estimated parameter $X$. The RBMTLS method incorporates such problems. 
	 
	 	In existing literature, all three solution techniques have been investigated to address the linear approximation problem \eqref{eq1.1} within the context of real case. For example, refer to \cite{golub1965numerical, golub1980analysis, yan2001solution} and the references therein. However, so far, only the LS solution technique has been examined in the reduced biquaternion domain. For instance, Zhang et al. \cite{zhang2020algebraic} investigated the least squares solutions to the reduced biquternion matrix equations $AXC= B$ and $AX= B$. To the best of our knowledge, RBTLS and RBMTLS solution techniques have not been explored. Notably, the RBMTLS method effectively encompasses both the RBLS and RBTLS methods, enhancing its overall applicability. In this work, we concentrate on real solutions to the linear approximation problem \eqref{eq1.1} in the reduced biquaternion domain. Here are the highlights of the work presented in this paper:
	\begin{itemize}
		\item The RBMTLS solution technique is discussed for obtaining the best approximate solution for a multidimensional overdetermined linear system $AX \approx B$. Furthermore, we investigate the conditions for the existence of a unique real RBMTLS solution and derive an explicit expression for the real RBMTLS solution.
		\item We propose the RBTLS and RBLS solution techniques. An RBMTLS problem transforms into an RBTLS problem when all columns of matrix $A$ are contaminated with noise. As a result, we utilize our developed RBMTLS solution technique to obtain RBTLS solutions. Additionally, we investigate the conditions for the existence of a unique real RBTLS solution and derive an explicit expression for the real RBTLS solution. Likewise, when all columns of matrix $A$ are error-free, the RBMTLS problem becomes an RBLS problem. Consequently, we utilize our developed RBMTLS solution technique to find real RBLS solutions.
		\item In light of complex matrix equations being a special case of reduced biquaternion matrix equations, we utilize our developed solution methods to find the best approximate solution to $AX \approx B$ over complex field.
	
		\end{itemize}
The manuscript is organized as follows. In Section \ref{sec2}, notation and preliminary results are presented. In Section \ref{sec3}, we present the solution techniques for RBMTLS, RBTLS, and RBLS problems. Finally, in Section \ref{sec4}, the numerical verification of our developed results is provided.
	\section{Notation and preliminaries}\label{sec2}
	\subsection{Notation}
	Throughout the paper, we denote $\R$, $\C$, and $\QR$ as the sets of all real numbers, complex numbers, and reduced biquaternion numbers, respectively. $\R^{m \times n}$, $\C^{m \times n}$, and $\QR^{m \times n}$ represent the sets of all $m \times n$ real, complex, and reduced biquaternion matrices, respectively.  For a matrix $A \in \R^{m \times n}$, the notation $A^{+}$ stands for the Moore-Penrose generalized inverse of $A$. $\mathcal{R}(A)$ denotes the column space of matrix $A$. The symbol $\norm{\cdot}_F$ represents the Frobenius norm. 
	Matlab command $randn(m,n)$ creates an $m \times n$ codistributed matrix of normally distributed random numbers whose every element is between $0$ and $1$. $rand(m,n)$ returns an $m \times n$ matrix of uniformly distributed random numbers. Let $A$ be any matrix of size $m \times n$. For any $i,j \in \{1,2,\ldots,n\}$ and $i<j$, $A(:,i)$ returns the $i^{th}$ column of matrix $A$, and $A(:,i : j)$ returns the submatrix with all rows of matrix $A$ and all columns from $i$ to $j$.\\
	 We use the following abbreviations throughout this paper:\\
RBMTLS: reduced biquaternion mixed least squares and total least squares, RBTLS: reduced biquaternion total least squares, RBLS: reduced biquaternion least squares, MTLS: mixed least squares and total least squares, TLS: total least squares, LS: least squares, SVD: singular value decomposition.
	\subsection{Preliminaries}
	A reduced biquaternion can be uniquely expressed as  $a= a_{0}+ a_{1}\,\i+ a_{2}\,\j+ a_{3}\,\k$,
	where $a_{i} \in \R$ for $i= 0, 1, 2, 3$, and $\i^2=  \k^2= -1, \; \j^2= 1$,
	$\i\j= \j\i= \k, \; \j\k= \k\j= \i, \; \k\i= \i\k= -\j$. The norm of $a$ is $\norm{a}= \sqrt{a_0^2+ a_1^2+ a_2^2+ a_3^2}$. The Frobenius norm for $A= (a_{ij}) \in \QR^{m \times n}$ is defined as follows:
	\begin{equation}\label{eq2.1}
		\norm{A}_F= \sqrt{\sum_{i= 1}^{m} \sum_{j= 1}^{n} \norm{a_{ij}}^{2}}.
	\end{equation}
	Let $A=A_0+ A_1\i+ A_2\j+ A_3\k \in \QR^{m \times n}$, where $A_t \in \R^{m \times n}$ for $t=0,1,2,3$. The real representation of matrix $A$, denoted as $A^R$, is defined as follows:
	\begin{equation}\label{eq2.2}
		A^R= \begin{bmatrix}
			A_0  &  -A_1  &  A_2  &  -A_3 \\
			A_1   &   A_0  &  A_3  &   A_2  \\
			A_2   &  -A_3 & A_0  &  -A_1  \\
			A_3  &  A_2  &  A_1  &  A_0 
		\end{bmatrix}.
	\end{equation}
	We have $$\norm{A}_F= \frac{1}{2} \norm{A^R}_F.$$ 
	For $a \in \R$, $A, B \in \QR^{m \times n}$, and $C \in \QR^{n \times t}$, we have $A= B \iff A^R= B^R$, $(A+B)^R= A^R+ B^R$, $(aA)^R= aA^R$, and $(AC)^R = A^R C^R$. Let
	\begin{equation*}
		Q_m= \begin{bmatrix}
			0      &   -I_m   &  0     & 0    \\
			I_m  &    0        &  0     & 0    \\
			0      &    0        &  0     & -I_m \\
			0      &    0        &  I_m & 0
		\end{bmatrix}, 
		R_m= \begin{bmatrix}
			0      &   0      &  I_m   & 0           \\
			0      &    0     &  0       &   I_m     \\
			I_m  &    0     &  0       & 0           \\
			0      &    I_m &  0       & 0
		\end{bmatrix}, \mbox{and} \;
		S_m= \begin{bmatrix}
			0      &   0      &  0       & -I_m           \\
			0      &    0     &  I_m    &  0    \\
			0      &  -I_m  &  0       & 0           \\
			I_m   &    0     &  0       & 0
		\end{bmatrix}.
	\end{equation*}
	Let $A_c^R$ denote the first block column of the block matrix $A^R$, i.e., $A_c^R= 
	[A_0^T, A_1^T, A_2^T, A_3^T]^T.$ We have
	\begin{equation}\label{eq2.3}
		A^R= [A_c^R, Q_m A_c^R, R_m A_c^R, S_m A_c^R].  
	\end{equation}
Let $a=a_0+ a_1\i + a_2\j + a_3\k \in \QR$ and $b = b_0+ b_1\i + b_2\j + b_3\k \in \QR$. We have $a=b \iff a_0=b_0, a_1=b_1, a_2=b_2,$ and $a_3=b_3$. We now establish two essential lemmas.
	\begin{lemma}\label{lem1}
		Let $A= A_0+ A_1\i+ A_2\j+ A_3\k \in \QR^{m \times n}$, $m> n$. Then matrix $A$ has full column rank if and only if matrix $A_c^R= 
		[A_0^T, A_1^T, A_2^T, A_3^T]^T \in \R^{4m \times n}$ has full column rank.
	\end{lemma}
\begin{proof}
	Let $A=(a_{ij})$, where $a_{ij}=a_{ij0}+a_{ij1}\i+a_{ij2}\j+a_{ij3}\k$. Let $v_j \in \QR^m$ denote the $j^{th}$ column of matrix $A$. The proof follows from the fact that the set of vectors $\{v_1,v_2,\ldots,v_n\}$ is linearly independent if the vector equation $x_1v_1 + 	x_2v_2 + \cdots +	x_nv_n=0$ has only the trivial solution $x_1=x_2=\cdots=x_n=0$, and by the equality property of reduced biquaternion numbers.
\end{proof}
	Using \eqref{eq2.1} and \eqref{eq2.2}, we can easily derive the following lemma . 
	\begin{lemma}\label{lem2}
		Let $A= A_0+ A_1\i+ A_2\j+ A_3\k \in \QR^{m \times n}$ and $B= B_0+ B_1\i+ B_2\j+ B_3\k \in \QR^{m \times d}$. Denote $A_c^R= 
		[A_0^T, A_1^T, A_2^T, A_3^T]^T$ and $B_c^R= 
		[B_0^T, B_1^T, B_2^T, B_3^T]^T$. Then the following properties hold.
		\begin{enumerate}[noitemsep,nolistsep]
			\item $\|[A, B]\|_F= \frac{1}{2}\|[A, B]^R\|_F$.
			\item $\|[A, B]^R\|_F= \|[A^R, B^R]\|_F$.
			\item $\|[A^R, B^R]\|_F= 2\|[A_c^R, B_c^R]\|_F$.
		\end{enumerate}
	\end{lemma}
	In the upcoming steps, we recall some well-known results that are helpful in establishing the main findings of this paper. For the scope of this paper, Eckart-Young-Mirsky matrix approximation theorem \cite{eckart1936approximation} has been rephrased to better suit our analysis. The modified version is as follows:
	\begin{lemma}\label{lem3}Let the SVD of $A \in \R^{m \times n}$ be given by $A= 	\bar{U}	\bar{\Sigma} 	\bar{V}^T$ with $r= \rank(A)$ and $k < r$. Let
		\begin{equation*}
				\begin{blockarray}{ccc@{}cc@{\hspace{4pt}}cl}
				&&&  \mLabel{k} & \mLabel{m-k} & &\\
				\begin{block}{cc@{\hspace{3pt}}c@{\hspace{9pt}}[cc@{\hspace{5pt}}c]l}
					&\bar{U}&=&\bar{U}_1,& \bar{U}_2& & \mLabel{m} \\
				\end{block}
			\end{blockarray}, \;
			\bar{\Sigma}= \begin{blockarray}{c@{}cc@{\hspace{4pt}}cl}
				&  \mLabel{k} & \mLabel{n-k} & &\\
				\begin{block}{[c@{\hspace{5pt}}cc@{\hspace{5pt}}c]l}
					&	\bar{\Sigma}_1& 0& & \mLabel{k} \\
					&0 & 	\bar{\Sigma}_2 & & \mLabel{m-k} \\
				\end{block}
			\end{blockarray}, \; \textrm{and} \; \;
				\bar{V}= 	\begin{blockarray}{c@{}cc@{\hspace{4pt}}cl}
				&  \mLabel{k} & \mLabel{n-k} & &\\
				\begin{block}{[c@{\hspace{5pt}}cc@{\hspace{5pt}}c]l}
					&	\bar{V}_{11}& 	\bar{V}_{12}& & \mLabel{k} \\
					&	\bar{V}_{21} & 	\bar{V}_{22} & & \mLabel{n-k} \\
				\end{block}
			\end{blockarray},
		\end{equation*}  where $	\bar{U} \in \R^{m \times m}$ and $	\bar{V} \in \R^{n \times n}$ are orthornormal matrices. Denote the diagonal matrices as $\bar{\Sigma}_1 = \diag{(	\bar{\sigma}_1, \ldots, 	\bar{\sigma}_k)}$ and $	\bar{\Sigma}_2 = \diag{(	\bar{\sigma}_{k+1}, \ldots, 	\bar{\sigma}_r)}$. If $A_k= [	\bar{U}_1, 	\bar{U}_2] \begin{bmatrix}
				\bar{\Sigma}_1 & 0 \\
			0  &  0
		\end{bmatrix} \begin{bmatrix}
				\bar{V}_{11} & 	\bar{V}_{12} \\
				\bar{V}_{21} & 	\bar{V}_{22}
		\end{bmatrix}^T = [	\bar{U}_1	\bar{\Sigma}_1	\bar{V}_{11}^T, 	\bar{U}_1	\bar{\Sigma}_1	\bar{V}_{21}^T]$, then
		\begin{equation*}
			\min_{\rank(B)=k}\|A-B\|_F= \|A-A_k\|_F= \sqrt{\sum_{i=k+1}^{r}\bar{\sigma}_{i}^2}.
		\end{equation*}
	\end{lemma}
In the above lemma, $A_k$ represents the best rank $k$ approximation of matrix $A$.
	\begin{lemma}\label{lem4}\cite{golub2013matrix}
		Let $A \in \R^{m \times n}$ and $B \in \R^{m \times d}$. Then, the solution to the real LS problem $$\min_{X}\norm{AX-B}_F$$ is $X=A^{+}B+(I-A^{+}A)Y$, where $Y$ is an arbitrary matrix of suitable size, and the least squares solution with the least norm is $X=A^+B.$
	\end{lemma}
	\section{An algebraic technique for RBMTLS problem}\label{sec3}
	In this section, we derive an algebraic solution technique for the RBMTLS problem by exploring the solution of the corresponding real MTLS problem. Suppose 
	\begin{equation}\label{eq3.1}
	A= A_0+ A_1\i+ A_2\j+ A_3\k \in  \QR^{m \times n} \; \; \textrm{and} \; \; B= B_0+ B_1\i+ B_2\j+ B_3\k \in \QR^{m \times d}.
	\end{equation}
	 Let the first $n_1$ columns of matrix $A$ be known exactly, and the remaining $n_2$ columns be contaminated by noise, where $n_1+n_2= n$.	
Let 
\begin{equation}\label{eq3.2}
A= [A_a, A_b] \; \; \textrm{and} \; \; X= [X_a^T, X_b^T]^T,
\end{equation}
where $A_a= A_{a0}+ A_{a1}\i+ A_{a2}\j+ A_{a3}\k \in \QR^{m \times n_1}$, $A_b=  A_{b0}+ A_{b1}\i+ A_{b2}\j+ A_{b3}\k \in \QR^{m \times n_2}$, and partitioning of $X$ is conformal with $A_a$ and $A_b$. We confine ourselves to the case when $m \geq n+d$ and $A_a$ has full column rank.\\
 The multidimensional RBMTLS problem can be formulated as:
\begin{equation}\label{eq3.3}
	\min_{X_a,X_b,\widehat{E}_b, \widehat{G}}\norm{[\widehat{E}_b, \widehat{G}]}_F \; \; \; \textrm{subject to} \; \; \;  A_aX_a+\left( A_b+\widehat{E}_b\right)X_b=B+\widehat{G}. 
\end{equation} 
Once a minimizing $[\widehat{E}_b, \widehat{G}]$ is found, then any $X=[X_a^T, X_b^T]^T$ which solves the corrected system in \eqref{eq3.3} is called the RBMTLS solution. 
\begin{remark}\label{rem1}
	Please note that by varying $n_1$ from $0$ to $n$, the above formulation can incorporate the RBTLS, RBMTLS, and RBLS problems. When $n_1=0$, we get the RBTLS problem. When $0<n_1<n$, we get the RBMTLS problem. When $n_1=n$, we get the RBLS problem.
\end{remark}

Let $C_a= [
A_{a0}^T, A_{a1}^T, A_{a2}^T, A_{a3}^T]^T \in \R^{4m \times n_1}$, $C_b= [
A_{b0}^T, A_{b1}^T, A_{b2}^T, A_{b3}^T]^T \in \R^{4m \times n_2}$, $C= [C_a, C_b]$, and $D= [B_0^T, B_1^T, B_2^T, B_3^T]^T \in \R^{4m \times d}$. Consider a multidimensional real MTLS problem 
\begin{equation}\label{eq3.4}
	\min_{X_a,X_b,\widetilde{E}_b, \widetilde{G}} \norm{[\widetilde{E}_b, \widetilde{G}]}_F \; \; \; \textrm{subject to} \; \; \; C_aX_a+\left( C_b+\widetilde{E}_b\right)X_b=D+\widetilde{G}. 
\end{equation}
Once a minimizing $[\widetilde{E}_b, \widetilde{G}]$ is found, then any $X=[X_a^T, X_b^T]^T$ which solves the corrected system in \eqref{eq3.4} is called the real MTLS solution.

\noindent
In the forthcoming results on the RBMTLS solution, we will be using the following notations: Let $\widetilde{E}_b= [E_{b0}^T, E_{b1}^T, E_{b2}^T, E_{b3}^T]^T \in \R^{4m \times n_2}$ and $\widetilde{G}= [G_0^T, G_1^T, G_2^T, G_3^T]^T \in \R^{4m \times d}$, where $E_{bt} \in \R^{m \times n_2}$ and $G_{t} \in \R^{m \times d}$ for $t=0,1,2,3.$ 
\begin{theorem}\label{thm1}
	Consider the RBMTLS problem \eqref{eq3.3} and the real MTLS problem \eqref{eq3.4}. Let $X=[X_a^T, X_b^T]^T$ be a real matrix. Then, $X$ is an RBMTLS solution if and only if $X$ is a real MTLS solution. In this case, if $X$ represents a real MTLS solution, then there exist $\widetilde{E}_b$ and $\widetilde{G}$ such that 
	\begin{equation*}
		\norm{[\widetilde{E}_b, \widetilde{G}]}_F= \min, \; \; C_aX_a+ (C_b+ \widetilde{E}_b)X_b= D+ \widetilde{G}.
	\end{equation*}
Let $\widehat{E}_b= E_{b0}+ E_{b1}\i+ E_{b2}\j+ E_{b3}\k \in \QR^{m \times n_2}$ and $\widehat{G}= G_0+ G_1\i+ G_2\j+ G_3\k \in \QR^{m \times d}$. Then,
	\begin{equation*}
		\norm{[\widehat{E}_b, \widehat{G}]}_F= \min, \; \;  A_aX_a+ (A_b+ \widehat{E}_b)X_b= B+ \widehat{G}.
	\end{equation*}
	Therefore, there exist $\widehat{E}_b$ and $\widehat{G}$ such that $X$ is an RBMTLS solution.
\end{theorem}
\begin{proof}
	If $X=[X_a^T, X_b^T]^T \in \R^{n \times d}$ is a real MTLS solution, then there exist real matrices $\widetilde{E}_b \in \R^{4m \times n_2}$ and $\widetilde{G} \in \R^{4m \times d}$ such that 
	\begin{equation*}
		\norm{[\widetilde{E}_b, \widetilde{G}]}_F= \textrm{min}, \; \;	[C_a, C_b+ \widetilde{E}_b]X= D+ \widetilde{G}.
	\end{equation*}
	We have
	\begin{align}\label{eq3.5}
		\left[[C_a, C_b+ \widetilde{E}_b], Q_m [C_a, C_b+ \widetilde{E}_b], R_m [C_a, C_b+ \widetilde{E}_b], S_m [C_a, C_b+ \widetilde{E}_b]\right]
		\begin{bmatrix}
			X & 0 & 0 & 0 \\
			0 & X & 0 & 0 \\
			0 & 0 & X & 0 \\
			0 & 0 & 0 & X
		\end{bmatrix} \nonumber &\\= \left[(D+ \widetilde{G}), Q_m (D+ \widetilde{G}), R_m (D+ \widetilde{G}), S_m (D+ \widetilde{G}) \right].
	\end{align}
	 Now,
	\begin{equation}\label{eq3.6}
		[C_a, C_b+ \widetilde{E}_b]= 
		\begin{bmatrix}
			A_{a0} & A_{b0}+ E_{b0} \\
			A_{a1} &  A_{b1}+ E_{b1} \\
			A_{a2} & A_{b2}+ E_{b2} \\
			A_{a3} & A_{b3}+ E_{b3} 
		\end{bmatrix}, \; D+\widetilde{G}= 
		\begin{bmatrix}
			B_0+ G_0 \\
			B_1+ G_1  \\
			B_2+ G_2 \\
			B_3+ G_3
		\end{bmatrix}.
	\end{equation}
	Construct the following reduced biquaternion matrices
	\begin{eqnarray*}
		\widehat{A}&:=& [A_{a0}, A_{b0}+E_{b0}]+ [A_{a1}, A_{b1}+E_{b1}]\i+ [A_{a2}, A_{b2}+E_{b2}]\j+ [A_{a3}, A_{b3}+E_{b3}]\k, \\
		\widehat{B}&:=& (B_0+ G_0)+ (B_1+ G_1)\i+ (B_2+ G_2)\j+ (B_3+ G_3)\k, \\
		\widehat{E}_b&:=& E_{b0}+ E_{b1}\i+ E_{b2}\j+ E_{b3}\k, \; \; \widehat{G}:= G_0+ G_1\i+ G_2\j+ G_3\k.
	\end{eqnarray*}
Using \eqref{eq2.2} and \eqref{eq3.6}, we have
\begin{equation*}
		\widehat{A}_c^R= [C_a, C_b+ \widetilde{E}_b], \;	\widehat{B}_c^R= D+\widetilde{G}, \; \left(\widehat{E}_b\right)_c^R=  \widetilde{E}_b, \; \textrm{and} \; \widehat{G}_c^R= \widetilde{G}.
\end{equation*}
	Using \eqref{eq2.2}, \eqref{eq2.3}, and \eqref{eq3.6}, we get
	\begin{eqnarray*}
		\widehat{A}^R&=& \left[[C_a, C_b+ \widetilde{E}_b], Q_m[C_a, C_b+ \widetilde{E}_b], R_m[C_a, C_b+ \widetilde{E}_b], S_m[C_a, C_b+ \widetilde{E}_b]\right], \\
		\widehat{B}^R&=& \left[(D+\widetilde{G}), Q_m(D+\widetilde{G}), R_m(D+\widetilde{G}), S_m(D+\widetilde{G})\right],\\
		X^R&=& 
		\begin{bmatrix}
			X & 0 & 0 & 0 \\
			0 & X & 0 & 0 \\
			0 & 0 & X & 0 \\
			0 & 0 & 0 & X
		\end{bmatrix}.
	\end{eqnarray*}
	Therefore, Equation \ref{eq3.5} is equivalent to
	\begin{eqnarray}
		\widehat{A}^R X^R&=& \widehat{B}^R, \label{eq3.7}\\
		(\widehat{A}X)^R&=& \widehat{B}^R, \nonumber\\
		\widehat{A}X&=& \widehat{B}. \label{eq3.8}
	\end{eqnarray}
	Now, 
	\begin{eqnarray}
		\widehat{A}&=& [A_{a0}, A_{b0}+ E_{b0}]+ [A_{a1}, A_{b1}+ E_{b1}]\i+ [A_{a2}, A_{b2}+ E_{b2}]\j+ [A_{a3}, A_{b3}+ E_{b3}]\k \nonumber \\
		&=&\left[(A_{a0}+A_{a1}\i+ A_{a2}\j+ A_{a3}\k), (A_{b0}+ A_{b1}\i+ A_{b2}\j+ A_{b3}\k)+ (E_{b0}+ E_{b1}\i+ E_{b2}\j+ E_{b3}\k)\right] \nonumber\\
		&=& [A_a, A_b+ \widehat{E}_b] \label{eq3.9}.
	\end{eqnarray}
	and
	\begin{eqnarray}
		\widehat{B}&=& (B_0+ G_0)+ (B_1+ G_1)\i+ (B_2+ G_2)\j+ (B_3+ G_3)\k \nonumber\\
		&=& (B_0+ B_1\i+ B_2\j+ B_3\k)+ (G_0 +G_1\i+ G_2\j+ G_3\k) = B+ \widehat{G}. \label{eq3.10}
	\end{eqnarray}
	Using \eqref{eq3.9} and \eqref{eq3.10}, Equation \ref{eq3.8} is equivalent to
	\begin{eqnarray}
		[A_a, A_b+\widehat{E}_b]X&=& B+ \widehat{G}, \nonumber \\ \nonumber
		[A_a, A_b+\widehat{E}_b]\begin{bmatrix}
			X_a \\
			X_b
		\end{bmatrix}&=& B+ \widehat{G}, \nonumber\\
		A_aX_a+ (A_b+ \widehat{E}_b)X_b &=& B+ \widehat{G} \label{eq3.11}.
	\end{eqnarray}
	Using Lemma \ref{lem2}, we can verify that
	\begin{equation}\label{eq3.12}
		\norm{[\widehat{E}_b, \widehat{G}]}_F= \frac{1}{2} \norm{[\widehat{E}_b, \widehat{G}]^R}_F= \frac{1}{2} \norm{[\widehat{E}_b^R, \widehat{G}^R]}_F=  \norm{[\widetilde{E}_b, \widetilde{G}]}_F= \min.  
	\end{equation}
	Combining \eqref{eq3.11} and \eqref{eq3.12}, we can conclude that there exist reduced biquaternion matrices $\widehat{E}_b \in \QR^{m \times n_2}$ and $\widehat{G} \in \QR^{m \times d}$ such that
$X=[X_a^T, X_b^T]^T  \in  \R^{n \times d}$ is an RBMTLS solution, and vice versa.
\end{proof}
Next, we derive an explicit expression for the real RBMTLS solution $X$. Perform $n_1$ Householder transformations using matrix $Q \in \R^{4m \times 4m}$ on the matrix $[C, D]$ such that
\vspace{-0.5cm}
\begin{equation}\label{eq3.13}
	Q^T [C, D]= Q^T [C_a, C_b, D]= 
	\begin{blockarray}{c@{}ccc@{\hspace{4pt}}cl}
		&  \mLabel{n_1} & \mLabel{n_2} & \mLabel{d} & &\\
		\begin{block}{[c@{\hspace{5pt}}ccc@{\hspace{5pt}}c]l}
			& R_{11} & R_{12} & R_{1d} & & \mLabel{n_1} \\
			& 0         & R_{22} & R_{2d} & & \mLabel{4m-n_1} \\
		\end{block}
	\end{blockarray}.
\end{equation}
Consider the partitioning of $Q$ as $Q= [Q_1, Q_2]$, where $Q_1 \in \R^{4m \times n_1}$ and $Q_2 \in \R^{4m \times (4m-n_1)}$. Let the SVD of $[R_{22}, R_{2d}]$ be given by 
	\begin{equation}\label{eq3.14}
		[R_{22}, R_{2d}]= U \Sigma V^T,
	\end{equation} where $U$ and $V$ are real orthonormal matrices, $\Sigma= \diag(\sigma_1, \sigma_2, \ldots, \sigma_{n_2+d})$, and the singular values of $[R_{22}, R_{2d}]$ satisfy 
\begin{equation}\label{eq3.15}
\sigma_1 \geq \sigma_2 \geq \ldots \geq \sigma_{n_2} > \sigma_{n_2+1} \geq \ldots \geq \sigma_{n_2+d} > 0.
\end{equation} Partition $U$, $\Sigma$, and $V$ as
	\vspace{-0.5cm}
	\begin{equation}\label{eq3.16}
		\begin{blockarray}{ccc@{}cc@{\hspace{4pt}}cl}
			&&&  \mLabel{n_2} & \mLabel{4m-n_1-n_2} & &\\
			\begin{block}{cc@{\hspace{3pt}}c@{\hspace{9pt}}[cc@{\hspace{5pt}}c]l}
				&U&=&U_1,& U_2& & \mLabel{4m-n_1} \\
			\end{block}
		\end{blockarray}, \; 
		\Sigma= \begin{blockarray}{c@{}cc@{\hspace{4pt}}cl}
			&  \mLabel{n_2} & \mLabel{d} & &\\
			\begin{block}{[c@{\hspace{5pt}}cc@{\hspace{5pt}}c]l}
				&\Sigma_1& 0& & \mLabel{n_2} \\
				&0 & \Sigma_2 & & \mLabel{4m-n_1-n_2} \\
			\end{block}
		\end{blockarray}, \;
		V= 	\begin{blockarray}{c@{}cc@{\hspace{4pt}}cl}
			&  \mLabel{n_2} & \mLabel{d} & &\\
			\begin{block}{[c@{\hspace{5pt}}cc@{\hspace{5pt}}c]l}
				&V_{11}& V_{12}& & \mLabel{n_2} \\
				&V_{21} & V_{22} & & \mLabel{d} \\
			\end{block}
		\end{blockarray}.
	\end{equation}
	In the following theorem, we present the conditions for the existence of a unique real RBMTLS solution, and in this case, provide an explicit expression for the real RBMTLS solution.
	\begin{theorem}\label{thm2}
		With the notations in \eqref{eq3.13} and \eqref{eq3.16}, consider the RBMTLS problem \eqref{eq3.3}. Let the SVD of $[R_{22}, R_{2d}]$ be as in \eqref{eq3.14}, and let its singular values be as in \eqref{eq3.15}. If $\sigma_{n_2} > \sigma_{n_2+1}$ and $V_{22}$ is nonsingular, then the real RBMTLS solution exists and is unique. In this case, the real RBMTLS solution is given by
		\begin{equation}\label{eq3.17}
			X= \begin{bmatrix}
				R_{11}^{-1}R_{1d} \\
				0
			\end{bmatrix}+
			\begin{bmatrix}
				R_{11}^{-1}R_{12} \\
				-I_{n_2}
			\end{bmatrix}V_{12} V_{22}^{-1}.
		\end{equation}
	
	\end{theorem}
	\begin{proof}
	Using Theorem \ref{thm1}, $X$ represents an RBMTLS solution if and only if $X$ is a real MTLS solution. Therefore, to find the RBMTLS solution, we find the real MTLS solution. Now, the real linear system corresponding to \eqref{eq1.1} is given by
		\begin{equation*}
			\left[C_a, C_b\right]\begin{bmatrix}
				X_a \\
				X_b
			\end{bmatrix} \approx D, \; \; \; 	\left[C_a, C_b, D\right]\begin{bmatrix}
				X_a \\
				X_b \\
				-I_d
			\end{bmatrix} \approx 0.
		\end{equation*}
	To find the real MTLS solution, we modify the above system in such a way that it becomes compatible. We achieve this by perturbing matrices $C_b$ and $D$ while keeping matrix $C_a$ exact, as in \eqref{eq3.4}. By pre-multiplying both sides of the above system by $Q^T$ and using \eqref{eq3.13}, we get
	\begin{eqnarray*}
		[Q_1, Q_2]^T [C_a, C_b, D]
		\begin{bmatrix}
			X_a \\
			X_b \\
			-I_d
		\end{bmatrix}\approx 0, \; \; \; 
		\begin{bmatrix}
			R_{11} & R_{12} & R_{1d} \\
			0         & R_{22} & R_{2d}
		\end{bmatrix}\begin{bmatrix}
			X_a \\
			X_b \\
			-I_d
		\end{bmatrix} \approx 0.
	\end{eqnarray*}
Let
	 $$R:= \begin{blockarray}{c@{}ccc@{\hspace{4pt}}cl}
		&  \mLabel{n_1} & \mLabel{n_2} & \mLabel{d} & &\\
		\begin{block}{[c@{\hspace{5pt}}ccc@{\hspace{5pt}}c]l}
			& R_{11} & R_{12} & R_{1d} & & \mLabel{n_1} \\
			& 0         & R_{22} & R_{2d} & & \mLabel{4m-n_1} \\
		\end{block}
	\end{blockarray}.$$
To make the above system compatible, the matrix $[X_a^T, X_b^T, -I_d]^T$ should be in the null space of $R$. Therefore, by the rank-nullity theorem, the rank of the matrix $R$ must be reduced to $n_1 +n_2$. We achieve this by modifying matrix $R$. To keep matrix $C_a$ exact, we modify matrix $R$ without perturbing the matrix $R_{11}$. 
 
 Now, matrix $A_a$ has full column rank. In view of Lemma \ref{lem1}, the matrix $C_a$ also has full column rank $n_1$, which implies that $R_{11}$ is a nonsingular upper triangular matrix. As a result, modifying $R_{12}$ and $R_{1d}$ does not affect the rank of the matrix $R$. Consequently, we do not modify these matrices. Instead, we modify matrices $R_{22}$ and $R_{2d}$. 
	
		Let $\widetilde{R}:= \begin{bmatrix}
			R_{11} & R_{12} & R_{1d} \\
			0         & \widetilde{R}_{22} & \widetilde{R}_{2d}
		\end{bmatrix}$ be the modified matrix such that the system $\widetilde{R}[X_a^T, X_b^T, -I_d]^T=0$ is compatible. Now our aim is to find $ \widetilde{R}_{22}$ and $\widetilde{R}_{2d}$. We first focus on the reduced real TLS problem $R_{22}X_b \approx R_{2d}.$ We have $$[R_{22}, R_{2d}] \begin{bmatrix}
	X_b\\
	-I_d
\end{bmatrix} \approx 0.$$ 
To find a solution to the reduced real TLS problem, the matrix $\left[X_b^T, -I_d\right]^T$ should be in the null space of $[R_{22}, R_{2d}]$. Therefore, by the rank-nullity theorem, the rank of the matrix $[R_{22}, R_{2d}]$ must be reduced to $n_2$. Let $[\widetilde{R}_{22}, \widetilde{R}_{2d}]$ denote the best rank $n_2$ approximation of $[R_{22}, R_{2d}]$. By Lemma \ref{lem3}, we have $$[\widetilde{R}_{22}, \widetilde{R}_{2d}]  = [U_1 \Sigma_1 V_{11}^T, U_1 \Sigma_1 V_{21}^T].$$ If $\sigma_{n_2}> \sigma_{n_2+1}$, then $[\widetilde{R}_{22}, \widetilde{R}_{2d}]$ represents the unique rank $n_2$ approximation of $[R_{22}, R_{2d}]$, and the columns of the matrix
		$\begin{bmatrix}
			V_{12} \\
			V_{22}
		\end{bmatrix}$ represent a basis for the null space of $[\widetilde{R}_{22}, \widetilde{R}_{2d}]$. We have
	\begin{equation*}
		[\widetilde{R}_{22}, \widetilde{R}_{2d}]\begin{bmatrix}
			V_{12} \\
			V_{22}
		\end{bmatrix}=0.
	\end{equation*}
If $V_{22}$ is nonsingular, then we get
	\begin{equation*}
	[\widetilde{R}_{22}, \widetilde{R}_{2d}]\begin{bmatrix}
		-V_{12}V_{22}^{-1} \\
		 -I_d
	\end{bmatrix}=0.
\end{equation*}
Hence, the reduced real TLS solution is unique and is given by $X_b= -V_{12}V_{22}^{-1}$. Notice that the rank of the modified matrix $\widetilde{R}$ is $n_1+n_2$.
		
		After computing $X_b$, we calculate $X_a$. We have $\widetilde{R}[X_a^T, X_b^T, -I_d]^T=0$. Since $R_{11}$ is nonsingular, we obtain a unique solution $X_a= R_{11}^{-1} (R_{1d}- R_{12}X_b)$. 
	\end{proof}
\begin{remark}\label{rem2}
The perturbation $\widetilde{E}_b$ to the matrix $C_b$ is given by $\widetilde{E}_b= \widetilde{C}_b- C_b,$ and the perturbation $\widetilde{G}$ to the matrix $D$ is given by $\widetilde{G}= \widetilde{D}-D$. We have
\begin{equation*}
[C_a, \widetilde{C}_b, \widetilde{D}]=	[Q_1, Q_2]\widetilde{R} = [Q_1, Q_2] \begin{bmatrix}
	R_{11} & R_{12} & R_{1d} \\
	0         & \widetilde{R}_{22} & \widetilde{R}_{2d}
\end{bmatrix}.
\end{equation*}  We obtain the perturbed matrices $\widetilde{C}_b:= Q_1 R_{12}+ Q_2 \widetilde{R}_{22}$ and $\widetilde{D}:= Q_1 R_{1d}+ Q_2 \widetilde{R}_{2d}$, where $\widetilde{R}_{22} = U_1 \Sigma_1 V_{11}^T$ and $\widetilde{R}_{2d} = U_1 \Sigma_1 V_{21}^T$. Now, we can obtain $\widehat{E}_b$ from $\widetilde{E}_b$ and $\widehat{G}$  from $\widetilde{G}$ using Theorem \ref{thm1}.
\end{remark}
	\noindent
	\textbf{Algebraic technique for RBTLS problem:}\\
	In the case where all columns of matrix $A$ are contaminated by noise (i.e., $n_1=0$ and $n_2=n$), the RBMTLS problem \eqref{eq3.3} simplifies to an RBTLS problem \eqref{eq1.3}. In this scenario, we have $A_a=0$ and $A_b=A$, as well as $C_a=0$ and $C_b=C$.
	Let $C= [
	A_{0}^T, A_{1}^T, A_{2}^T, A_{3}^T]^T \in \R^{4m \times n}$. Consider a multidimensional  real TLS problem 
	\begin{equation}\label{eq3.18}
		\min_{X,\widetilde{E}, \widetilde{G}} \norm{[\widetilde{E}, \widetilde{G}]}_F \; \; \; \textrm{subject to} \; \; \;  (C+\widetilde{E})X=D+\widetilde{G}.
	\end{equation}
Once a minimizing $[\widetilde{E}, \widetilde{G}]$ is found, then any $X$ which solves the corrected system in \eqref{eq3.18} is called the real TLS solution.

\noindent
In the forthcoming results on the RBTLS solution, we will be using the following notations: Let $\widetilde{E}= [E_{0}^T, E_{1}^T, E_{2}^T, E_{3}^T]^T \in \R^{4m \times n}$ and $\widetilde{G}= [G_0^T, G_1^T, G_2^T, G_3^T]^T \in \R^{4m \times d}$, where $E_t \in \R^{m \times n}$ and $G_t \in \R^{m \times d}$ for $t=0,1,2,3.$  In the following corollary, we provide the solution technique for the RBTLS problem \eqref{eq1.3}. 
	\begin{corollary}\label{cor1}
		Consider the RBTLS problem \eqref{eq1.3} and the real TLS problem \eqref{eq3.18}. Let $X$ be a real matrix. Then, $X$ is an RBTLS solution if and only if $X$ is a real TLS solution. In this case, if $X$ represents a real TLS solution, then there exist $\widetilde{E}$ and $\widetilde{G}$ such that 
	\begin{equation*}
		\norm{[\widetilde{E}, \widetilde{G}]}_F= \min, \; \;  (C+ \widetilde{E})X= D+ \widetilde{G}.
	\end{equation*}
Let $\widehat{E}= E_{0}+ E_{1}\i+ E_{2}\j+ E_{3}\k \in \QR^{m \times n}$ and $\widehat{G}= G_0+ G_1\i+ G_2\j+ G_3\k \in \QR^{m \times d}$. Then,
	\begin{equation*}
		\norm{[\widehat{E}, \widehat{G}]}_F= \min, \; \; (A+ \widehat{E})X= B+ \widehat{G}.
	\end{equation*}
	Therefore, there exist $\widehat{E}$ and $\widehat{G}$ such that $X$ is an RBTLS solution. 
	\end{corollary}
	\begin{proof}
		By taking $n_1= 0$ and $n_2= n$, the proof follows similar to the proof of Theorem \ref{thm1}.
	\end{proof}
	Next, we derive an explicit expression for the real RBTLS solution $X$. By taking $n_1=0$ and $n_2=n$, equations \eqref{eq3.13}, \eqref{eq3.14}, \eqref{eq3.15}, and \eqref{eq3.16} are simplified to
	\begin{eqnarray}\label{eq3.19}
		\begin{blockarray}{ccc@{}cc@{\hspace{4pt}}cl}
			&&& \mLabel{n} & \mLabel{d} & &\\
			\begin{block}{cc@{\hspace{3pt}}c@{\hspace{9pt}}[cc@{\hspace{5pt}}c]l}
			&Q^T [C, D]&=	& R_{22} & R_{2d} & & \mLabel{4m} \\
			\end{block}
		\end{blockarray}.
	\end{eqnarray}
Thus, the SVD of $[R_{22}, R_{2d}]$ is given by 
	\begin{equation}\label{eq3.20}
		[R_{22}, R_{2d}]= U \Sigma V^T,
	\end{equation} where $U$ and $V$ are real orthonormal matrices, $\Sigma= \diag(\sigma_1, \sigma_2, \ldots, \sigma_{n+d})$, and the singular values of $[R_{22}, R_{2d}]$ satisfy 
\begin{equation}\label{eq3.21}
\sigma_1 \geq \sigma_2 \geq \ldots \geq \sigma_{n} > \sigma_{n+1} \geq \ldots \geq \sigma_{n+d} > 0.
\end{equation} Partition $U$, $\Sigma$, and $V$ as
	\vspace{-0.5cm}
	\begin{equation}\label{eq3.22}
		\begin{blockarray}{ccc@{}cc@{\hspace{4pt}}cl}
			&&&  \mLabel{n} & \mLabel{4m-n} & & \\
			\begin{block}{cc@{\hspace{3pt}}c@{\hspace{9pt}}[cc@{\hspace{5pt}}c]l}
			&U&=	&U_1& U_2& & \mLabel{4m} \\
			\end{block}
		\end{blockarray}, \; 
		\Sigma= \begin{blockarray}{c@{}cc@{\hspace{4pt}}cl}
			&  \mLabel{n} & \mLabel{d} & &\\
			\begin{block}{[c@{\hspace{5pt}}cc@{\hspace{5pt}}c]l}
				&\Sigma_1& 0& & \mLabel{n} \\
				&0 & \Sigma_2 & & \mLabel{4m-n} \\
			\end{block}
		\end{blockarray}, \; 
		V= 	\begin{blockarray}{c@{}cc@{\hspace{4pt}}cl}
			&  \mLabel{n} & \mLabel{d} & &\\
			\begin{block}{[c@{\hspace{5pt}}cc@{\hspace{5pt}}c]l}
				&V_{11}& V_{12}& & \mLabel{n} \\
				&V_{21} & V_{22} & & \mLabel{d} \\
			\end{block}
		\end{blockarray}.
	\end{equation}
	In the following corollary, we present the conditions for the existence of a unique real RBTLS solution, and in this case, provide an explicit expression for the real RBTLS solution. 
	\begin{corollary}\label{cor2}
		With the notations in \eqref{eq3.19} and \eqref{eq3.22}, consider the RBTLS problem \eqref{eq1.3}. Let the SVD of $[R_{22}, R_{2d}]$ be as in \eqref{eq3.20}, and let its singular values be as in \eqref{eq3.21}. If $\sigma_{n} > \sigma_{n+1}$ and $V_{22}$ is nonsingular, then the real RBTLS solution exists and is unique. In this case, the real RBTLS solution is given by
		\begin{equation}\label{eq3.23}
			X= -V_{12} V_{22}^{-1}.
		\end{equation}
	\end{corollary}
	\begin{proof}
			By taking $n_1= 0$ and $n_2= n$, the proof follows similar to the proof of Theorem \ref{thm2}.
	\end{proof}
\begin{remark}\label{rem3}
		The perturbation $\widetilde{E}$ to the matrix $C$ is given by $\widetilde{E}= \widetilde{C}- C,$ and the perturbation $\widetilde{G}$ to the matrix $D$ is given by $\widetilde{G}= \widetilde{D}-D.$ We have
		\begin{equation*}
			[\widetilde{C}, \widetilde{D}]=Q[\widetilde{R}_{22}, \widetilde{R}_{2d}].
		\end{equation*}
	We get the perturbed matrices $\widetilde{C}:= Q \widetilde{R}_{22}$ and $\widetilde{D}:= Q \widetilde{R}_{2d}$, where $ \widetilde{R}_{22}= U_1 \Sigma_1 V_{11}^T$ and $\widetilde{R}_{2d}= U_1 \Sigma_1 V_{21}^T.$ Now, we can obtain $\widehat{E}$ from $\widetilde{E}$ and $\widehat{G}$ from $\widetilde{G}$ using Corollary \ref{cor1}.
\end{remark}
   \noindent
	\textbf{Algebraic technique for RBLS problem:}\\
	When all columns of matrix A are error-free, i.e., $n_1 = n$ and $n_2 = 0$, the RBMTLS problem \eqref{eq3.3} becomes an RBLS problem \eqref{eq1.2}. In this scenario, we have $A_a=A$ and $A_b=0$, as well as $C_a=C$ and $C_b=0$.
	Let $C= [
	A_{0}^T, A_{1}^T, A_{2}^T, A_{3}^T]^T \in \R^{4m \times n}$. Consider a multidimensional real LS problem 
	\begin{equation}\label{eq3.24}
		\min_{X}\norm{CX-D}_F.
	\end{equation}
	In the following corollary, we provide the solution technique for RBLS problem \eqref{eq1.2}.
	\begin{corollary}\label{cor3}
	Consider the RBLS problem \eqref{eq1.2} and the real LS problem \eqref{eq3.24}.	Let $X$ be a real matrix. Then, $X$ is an RBLS solution if and only if $X$ is a real LS solution. In this case, the solution $X$ is given by
		\begin{equation}\label{eq3.25}
			X= C^{+}D+(I-C^{+}C)Z,
		\end{equation}
		where $Z$ is an arbitrary matrix of suitable size and the least squares solution with the least norm is $X= C^{+}D.$
	\end{corollary}
	\begin{proof}
		By taking $n_1= n$ and $n_2= 0$ in Theorem \ref{thm1}, we get that $X$ is an RBLS solution if and only if $X$ is a real LS solution. Using Lemma \ref{lem4}, we get the desired expression for the solution $X$.
	\end{proof}
	
	The results developed in this section can also be applied to several other special cases. The following remarks are in order.
	\begin{remark}
		When $d=1$, our results also include single-right-hand-side RBMTLS, RBTLS, and RBLS problems.
	\end{remark}
	\begin{remark}
	Complex matrix equations are special cases of reduced biquaternion matrix equations. Hence, our developed solution techniques are well-suited for finding the best approximate solution to $AX \approx B$ over complex fields.
		
		We take the real representation of matrix $A=A_0+ A_1\i  \in \C^{m \times n}$, where $A_t \in \R^{m \times n}$ for $t=0,1$, denoted by $\widetilde{A}^R$ as
		\begin{equation*}
			\widetilde{A}^R= \begin{bmatrix}
				A_0  &  -A_1  \\
				A_1   &   A_0  
			\end{bmatrix}.
		\end{equation*}
		Let $\widetilde{Q}_m= \begin{bmatrix}
			0      &   -I_m  \\
			I_m  &    0       
		\end{bmatrix}.$ Let $\widetilde{A}_c^R$ denote the first block column of the block matrix $\widetilde{A}^R$ i.e. $\widetilde{A}_c^R= [A_0^T, A_1^T]^T.$ We have $\widetilde{A}^R= [\widetilde{A}_c^R, \widetilde{Q}_m \widetilde{A}_c^R].$ By taking $\widetilde{A}^R$, $\widetilde{Q}_m$, and $\widetilde{A}_c^R$ as above, we can obtain results to solve the complex LS, TLS, and MTLS problems.
	\end{remark}
	\section{Numerical Verification}\label{sec4}
In this section, we present a numerical example to verify our findings. All calculations are performed on an Intel Core $i7-9700@3.00GHz/16GB$ computer using MATLAB $R2021b$ software.

	\begin{exam} \label{ex4.1}
	Let $F= F_0+F_1\i+F_2\j+F_3\k \in \QR^{m \times 50}$ $(m>50)$, where $F_i= randn(m,50) \in \R^{m \times 50}$ for $i=0,1,2,3$.	Let $X_0 = randn(50,35) \in \R^{50 \times 35}$. Take $G=FX_0$. Clearly, the reduced biquaternion matrix equation $FX = G$ is consistent, and $X_0$ is its solution. We partition the matrix $F$ as $F= [F_a, F_b],$ where $F_a \in \QR^{m \times 20}$ and $F_b \in \QR^{m \times 30}$. 
	
To assess the effectiveness of our proposed solution techniques in finding the best approximate solution to an inconsistent linear system, we intentionally introduce errors into the entries of matrices $F$ and $G$. This makes our original system inconsistent.
	
	 Let the error terms be denoted as $dA \in \QR^{m \times 20}$, $dB \in \QR^{m \times 30}$, and $dG \in \QR^{m \times 35}$. We represent the modified matrices as $A_a=F_a+dA$, $A_b=F_b+dB$, and $B=G+dG$. Consequently, we have an overdetermined linear system: 
	$$AX \approx B,$$ where $A=[A_a, A_b] \in \QR^{m \times 50}$ and $B \in \QR^{m \times 35}$ are known, and $X \in \R^{50 \times 35}$ is unknown. Now we will consider three different cases. In the first case, errors are introduced in matrices $F_b$ and $G$. In the second case, errors are introduced in matrices $F_a$, $F_b$, and $G$. Lastly, in the third case, errors are introduced only in matrix $G$.
	\begin{enumerate}[label=\textup{(\arabic*)}]
	\item First case:\\
	 Take $R = rand(65,65)$ and $E= 0.01\left(rand(m,65)R\right)$. Let 
	 \begin{eqnarray*}
	 	dA&=&0, \\
	 	dB &=& E(:,1:30)+E(:,1:30)\i+E(:,1:30)\j+E(:,1:30)\k, \\
	 	dG &=& E(:,31:65)+E(:,31:65)\i+E(:,31:65)\j+E(:,31:65)\k.
	 \end{eqnarray*} 
 We have $A_a = F_a, A_{b} = F_{b} + dB$, and $B = G + dG$. 
	\item Second case: \\
	Take $R = rand(85,85)$ and $E= 0.01\left(rand(m,85)R\right)$. Let 
	\begin{eqnarray*}
		dA &=& E(:,1:20)+E(:,1:20)\i+E(:,1:20)\j+E(:,1:20)\k, \\
		dB &=& E(:,21:50)+E(:,21:50)\i+E(:,21:50)\j+E(:,21:50)\k,\\
		 dG &=& E(:,51:85)+E(:,51:85)\i+E(:,51:85)\j+E(:,51:85)\k.
		  \end{eqnarray*}
		 We have $A_a = F_a + dA$, $A_b = F_b + dB$, and $B = G + dG.$ 
	\item Third case:\\
	Take $R = rand(35,35)$ and $E= 0.01\left(rand(m,35)R\right)$. Let 
	\begin{eqnarray*}
	dA&=&0,\\
	 dB&=&0, \\
	 dG &=& E+E\i+E\j+E\k. 
	\end{eqnarray*} 
We have $A_a = F_a$, $A_b=F_b$, and $B = G + dG$. 
	\end{enumerate}
In each of the three different cases, due to errors in matrices $A$ and $B$, there is no exact solution for system $AX \approx B$, so an approximate one is sought. In this example, we calculate the RBMTLS solution $(X_M)$, the RBTLS solution $(X_T)$, and the RBLS solution $(X_L)$ for the inconsistent system $AX \approx B$ in all three cases. 
	
	\textbf{Note:} To achieve the highest possible accuracy in the estimated $X$, it is important to obtain the solution by eliminating any errors that might be present in the entries of matrices $A$ and $B$. In all three cases, if we remove all errors from matrices $A$ and $B$, they become matrices $F$ and $G$, respectively. Therefore, $X_0$ represents the most accurate approximate solution for the system $AX \approx B$ in all three cases.
	
	Now we will calculate $X_M$, $X_T$, and $X_L$ using Theorem \ref{thm2}, Corollary \ref{cor2}, and Corollary \ref{cor3}, respectively. Let the errors be denoted as $\epsilon_1 = \|X_M-X_0\|_F$, $\epsilon_2 = \|X_T-X_0\|_F$, and $\epsilon_3 = \|X_L-X_0\|_F$. In this example, $m$ is any arbitrary number. We calculate errors $\epsilon_1$, $\epsilon_2$, and $\epsilon_3$ for different values of $m$. Since the input matrices are randomly generated, we calculate $\epsilon_1$, $\epsilon_2$, and $\epsilon_3$ by taking the average of errors obtained by solving our example twenty times for each $m$. 
\end{exam}

Figure \ref{fig} present comparisons for case $1$, case $2$, and case $3$, respectively, between $\epsilon_1$, $\epsilon_2$, and $\epsilon_3$. These comparisons are obtained by taking different values of $m$. For all values of $m$, we observe that in case $1$, $\epsilon_1 < \epsilon_2 < \epsilon_3$, while in case $2$, $\epsilon_2 < \epsilon_3 < \epsilon_1$. Lastly, in case $3$, $\epsilon_3 < \epsilon_2 < \epsilon_1$.
	\begin{figure}[H]
	\centering
	\subfigure[]{\includegraphics[width=0.51\textwidth]{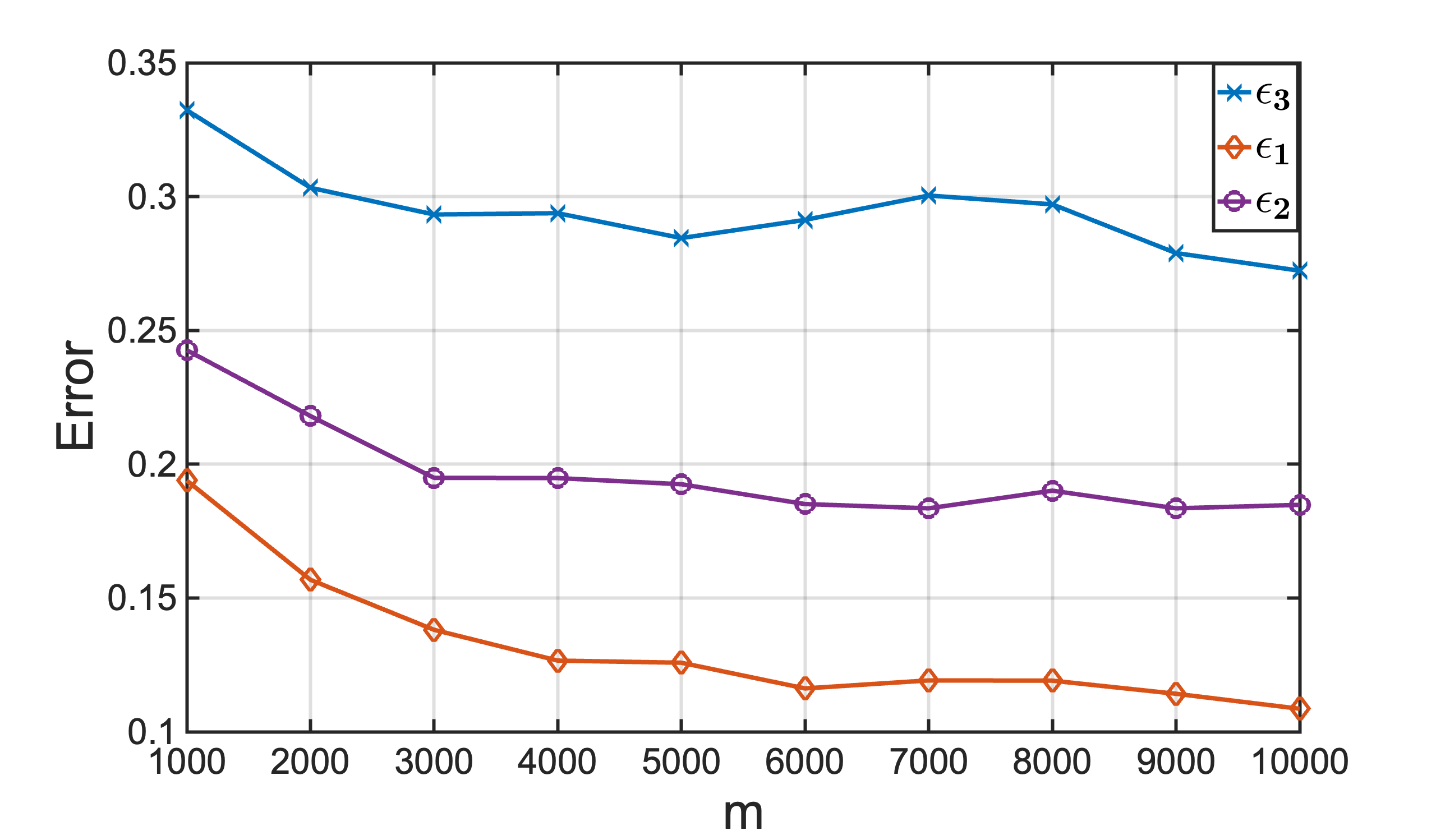}} 
	\hspace{-0.7cm}
	\subfigure[]{\includegraphics[width=0.51\textwidth]{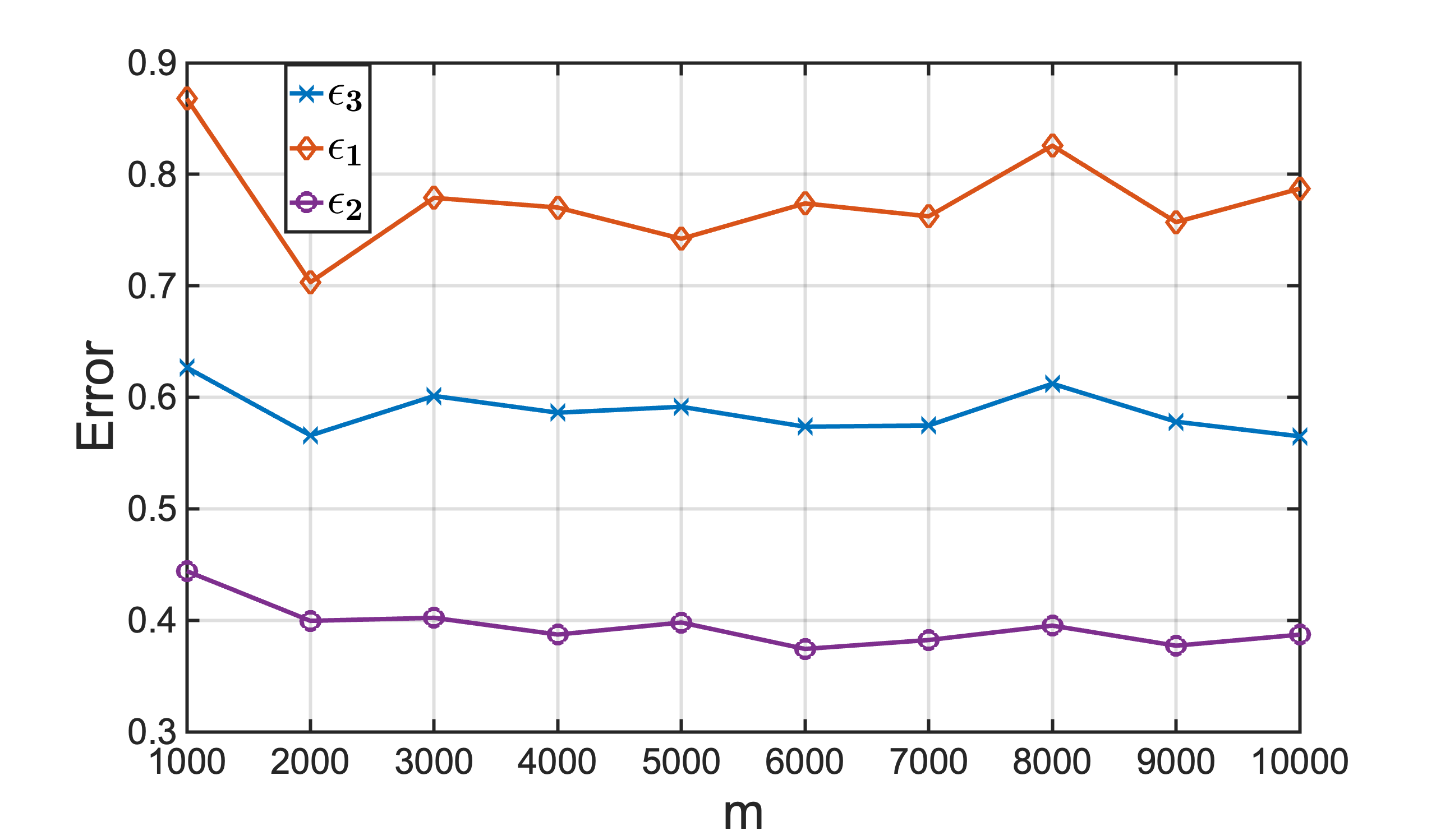}} 
	\subfigure[]{\includegraphics[width=0.51\textwidth]{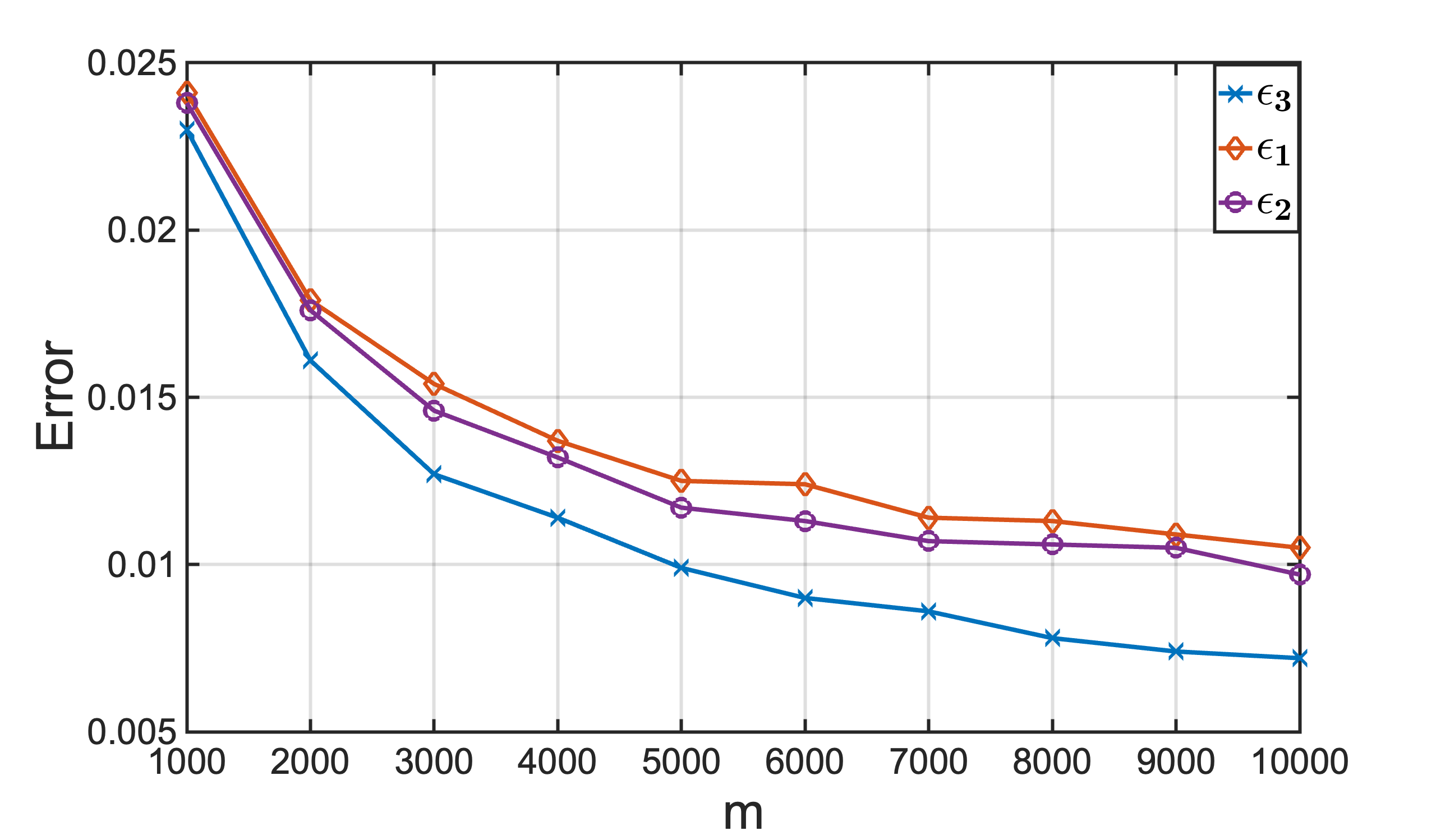}}
	\caption{The errors from the three solution techniques for (a) Case $1$ (b) Case $2$ (c) Case $3$}
	\label{fig}
\end{figure}
We conclude this example with the following remark:
	\begin{remark}
			\begin{enumerate}[label=\textup{(\arabic*)}]
		\item If there is an error in matrix $B$ along with a few columns of matrix $A$, then the RBMTLS solution technique offers the most accurate approximate solution to the overdetermined system $AX \approx B$.
		\item In cases where errors are present in both matrix $A$ and matrix $B$, the RBTLS solution technique yields the most accurate approximate solution to the overdetermined system $AX \approx B$.
		\item When the error is solely present in matrix $B$, the RBLS solution technique provides the most accurate approximate solution to the overdetermined system $AX \approx B$.
		\end{enumerate}
	\end{remark}
	\section{Conclusions} \label{sec5}
We have introduced a method to find the best approximate solution for an inconsistent linear system that arises in commutative quantum theory. An algebraic solution technique has been presented to address the RBMTLS problem. Our approach involved transforming the RBMTLS problem into a real MTLS problem using the real representation of reduced biquaternion matrices. This equivalence enabled us to deduce conditions for the existence of a unique real RBMTLS solution and provide expressions for the real RBMTLS solution. We have also proposed the RBTLS and RBLS solution techniques, which can be regarded as special cases of the RBMTLS solution technique. Additionally, our developed results encompass single-right-hand-side RBMTLS, RBTLS, and RBLS problems as special cases. Furthermore, the developed solution methods are applied to find the best approximate solution to the linear system $AX \approx B$ over complex field, showcasing their versatility in handling complex matrix equations, which are special cases of reduced biquaternion matrix equations. Finally, we have given a numerical example to verify our results. The developed methods for solving overdetermined linear systems will have potential applications in the study of digital signals and image processing in commutative quaternionic theory.
	
\section*{Acknowledgement}
	The second author is thankful to the Government of India for providing financial support through Prime Minister's Research Fellowship (PMRF). 
	\bibliography{p2lama}
\end{document}